\documentclass{article}
\usepackage[utf8]{inputenc}
\usepackage[T1]{fontenc}
\usepackage{amsmath}
\usepackage{amsfonts}
\usepackage{enumitem}
\usepackage{bbm}
\usepackage[letterpaper,hmargin=1in,vmargin=1.25in]{geometry}
\usepackage{multirow}
\usepackage{comment}
\usepackage{mathtools}
\usepackage{authblk}
\usepackage{tikz}   
\usepackage{nicefrac}
\usepackage{graphicx}
\usepackage{amssymb,latexsym,xspace}
\usepackage{hyperref}
\usepackage{microtype}
\usepackage{doi}
\usepackage[scr=esstix,cal=boondox]{mathalfa} 

\newtheorem{theorem}{Theorem}
\newtheorem{lemma}[theorem]{Lemma}

\newtheorem{corollary}[theorem]{Corollary}

\newenvironment{proof}
{\textit{\bf Proof:} }
{$\hfill\square$\\ \\}

\newenvironment{proofofthm}[1]{\textit{\bf Proof of theorem #1:} }{$\hfill\square$\\ \\}

\newcommand{\given}{{\;|\;}}
\newcommand{\prob}{\mathbf{P}}
\newcommand{\E}{\mathbf{E}}
\newcommand{\St}{\widetilde{S}}

\newcommand{\supp}{{\rm support}}
\newcommand{\degree}{{\rm degree}}

\def \half {{\textstyle \frac{1}{2}}}
\def \ind {{ \mathbf{1}}}
\def \sets {{ \mathcal S}}
\def \P {{ \mathbf{P}}}
\def \Z {{ \mathbf{Z}}}
\def \tower {{\uparrow \uparrow}}
\def \diam {{\rm diameter}}

\pgfdeclareradialshading[droplet color]{droplet}{\pgfqpoint{-10bp}{-10bp}}{%
	color(0bp)=(droplet color!10!white);
	color(50bp)=(droplet color!10!white)
}
\colorlet{droplet color}{black}
\tikzset{%
	raindrop/.pic={
		code={\tikzset{scale=1/10}
			\shade [shading=droplet]
			(0,0)  .. controls ++(0,-1) and ++(0,1) .. (1,-2)
			arc (360:180:1)
			.. controls ++(0,1) and ++(0,-1) .. (0,0) -- cycle;
}}}

\DeclarePairedDelimiter{\ceil}{\lceil}{\rceil}

\title{Transience of Simple Random Walks With Linear Entropy Growth}
\author[1]{Ben Morris}
\author[2]{Hamilton Samraj Santhakumar}
\affil[1]{Department of Mathematics, University of California, Davis}
\affil[2]{Department of Mathematics, University of California, Davis}
\date{February 15, 2023}

\begin{document}

\maketitle

\begin{abstract}
Using the technique of evolving sets, we explore the connection between entropy growth and transience for simple random walks on connected
  infinite graphs with bounded degree. In particular we show that for a simple random walk starting at a vertex $x_0$, if the entropy after $n$ steps, $E_n$
  is at least $Cn$ where the $C$ is independent of $x_0$, then the random walk is transient. We also give an example which demonstrates that the condition
  of $C$ being independent of $x_0$ is necessary.
\end{abstract}

\section{Introduction}
Let $X_1, X_2, \dots$ be simple random walk on an infinite graph of maximum degree $d$ and let $p$ be
its transition matrix. Let $V$ denote the set of vertices of this graph.
The main result of this paper is Theorem \ref{mainthm} given below which gives a connection between the entropy growth of the random walk 
and its transience.

\begin{theorem}\label{mainthm}
  Suppose that $X_0 = x_0$. Let $E_n$ be the entropy of $X_n$, i.e., $E_n = -\sum_{x\in V} \prob(X_n=x)\log\big(\prob(X_n=x)\big)$. If
	\begin{equation}\label{linearENT}
	E_n\geq Cn  
	\end{equation}
	for some $C$ independent of $x_0$, then the random walk is transient.
      \end{theorem}
This problem was suggested to us by Itai Benjamini.\\

 Note that the entropy defined in the above theorem is a finite sum since each vertex has finite degree and hence the
support of $X_n$ is finite. So the entropy makes perfect sense and there is no question of convergence. Henceforth whenever
the inequality $(\ref{linearENT})$ holds for some $C$ independent of the starting position, we will say that the $\textit{linear entropy condition}$ holds.
To see why the bounded degree assumption is necessary, note
that if we append $b_k$ pendant edges to vertex $k$ in $\Z^+$ (the graph of positive integers)
then the resulting graph is still recurrent, and if the sequence $b_k$ grows
fast enough then the walk satisfies the linear growth condition.  
In Section \ref{Example} we give an example (told to us by Gady Kozma) to show that the theorem fails to hold if the $C$ in (\ref{linearENT}) is not independent of the starting
position of the random walk.

\section{Entropy and the Probability of Escape in $n$ Steps}
\subsection{Background and notation}
For any set $V'\subset V$,
define $p^n(x, V') := \sum_{v\in V'}p^n(x,v)$, with a similar definition for $p(V', x)$.
For real numbers $a$ and $b$, define $a \wedge b := \min(a,b)$.
Our proof of Theorem \ref{mainthm} uses a set valued process called the evolving set process, which was
used in \cite{evolvingsets} to obtain bounds on mixing times of Markov chains in terms of isoperimetric inequaltiies. The notion of evolving sets is related to strong stationary duals introduced by Fill and Diaconis \cite{strongstatdual}. Before we introduce evolving sets, let's make the following observation. Recall that $d$ is the maximum degree of $G$.
\begin{lemma}
  Assume that the linear entropy growth condition holds and that $X_0=x_0$. Then for any set $A\subset V$
  and integer $n \geq 1$, we have 
	\begin{equation}\label{ProbLeakage}
	p^n(x_0,A^c)\geq \frac{Cn - \log2|A|}{n\log d}.
	\end{equation}
\end{lemma}
\begin{proof}
  Fix $A \subset V$ and define $A^c := \supp(X_n) \setminus A$.
  For $i \in V$ define $p_i := p^n(x_0, i)$ and define
  \[
    q := p^n(x_0, A^c) = \sum_{i \in A^c} p_i
  \]

  Note that, by concavity of the function $x \mapsto -x \log x$, the average value of $-p_i \log p_i$ over the set $A$ is at most
  $-{1-q \over |A|} \log{1-q \over |A|}$. Furthermore, the average value of $-p_i \log p_i$ over the set $A^c$ is at most
  $-{q \over |A^c| }\log{q \over |A^c|}$. Hence the entropy satisfies
       \begin{align}
\nonumber         E_n &\leq - (1 - q) \log {1 - q \over |A|}  - q  \log {q \over |A^c|}    \\
                      &= [-(1-q) \log(1-q) - q \log q] + (1-q) \log |A| + q \log |A^c| \,.  \label{chain}
       \end{align}
       (Equation (\ref{chain}) also follows from the entropy chain rule.) By concavity
       of $x \mapsto -x \log x$, the term in square brackets is at most $\log 2$. Combining this with the linear
       entropy growth condition gives
       \[
       Cn \leq  E_n \leq \log 2 + q( \log |A^c| - \log |A|) + \log |A|.
       \]
       Rearranging terms gives
       \[
         q \geq {Cn - \log |A| - \log 2  \over \log |A^c| - \log |A|} \,.
         \]
         Since $|A^c| \leq |\supp(X_n)| \leq d^n$ and $\log |A| \geq 0$, the lemma follows.
\end{proof}

Suppose that $\pi$ is a stationary measure
for the transition kernel $p$. 
For integers $t \geq 0$ define $Q_t(x,y)=\pi(x) p^t (x,y)$. For $B\subset V$, define
$Q_t(B,y) := \sum_{x\in B} Q_t (x,y)$. Then we have the following corollary.

\begin{corollary} \label{leave}
  For any finite set $S \subset V$ we have
  \[
    Q_n(S, A^c) \geq \pi(S) \frac{Cn - \log2|A|}{n\log d}.
  \]
  \end{corollary}

\section{The Evolving Set Process}\label{EvolvingDef}

\textbf{Definition (Evolving Sets).}
We now define a process on subsets of $V$, which is similar to the evolving sets of \cite{evolvingsets}, except that
here the sets are only defined at certain times $T_0 \leq T_1 \leq T_2 \leq \cdots$.
Fix (the starting state) $S \subset V$. We define a nonnegative integer valued process $T_0, T_1, \ldots$ and a set valued
process $S_{T_0}, S_{T_1}, \ldots$ inductively as follows.
Let $U_1, U_2, \ldots$ be i.i.d~Uniform$[0,1]$ random variables.
Set $T_0 = 0$ and $S_{T_0}=S$.
Now assuming $T_{m-1}$ and $S_{T_{m-1}}$ are given, define $T_{m}= T_{m-1} + L_m$, where $L_1, L_2, \ldots$ is a positive integer valued process
such that $L_j$ is a function of $S_{T_{j-1}}$. Then define
\[
  S_{T_m} =  \{y: Q_{L_m}(S_{T_{m-1}},y)\geq U_m \pi(y)\}.
\]
The construction used here is similar to the one in $\cite{evolvingsets}$. However, here we
consider many steps of Markov chain at  a time. 
We shall call the process $S_{T_0}, S_{T_1}, \ldots$ the {\it intermittant evolving set process.} \\

Note that the intermittant evolving set process
depends on the sequence of random variables $L_1, L_2, \ldots$ (the ``time gaps'').
The results in this section are true for any choice of time gaps.
(From Section \ref{RelatingTransienceandEvolvingSets} onwards, we choose a particular sequence $\big( L_j\big)_{j\geq 1}$.) For every integer $t\geq 0$, define 
\begin{equation*}
	a(t) = \max \{i\ : \ T_i \leq t \}.
\end{equation*}
Then we have the following lemma.


\begin{lemma}\label{ProbPowersLemma}
	For every nonnegative integer $t$, we have
	\begin{equation}
	Q_t(S,y) = \E \big[ Q_{t-T_{a(t)}}(S_{T_{a(t)}},y) \big].
	\end{equation}
\end{lemma}
\begin{proof}
For integers $n$ with $0 \leq n \leq t$, 
define $a_n(t) := a(t) \wedge n$. Note that $a_0(t) = 0$. Furthermore, since $T_t \geq t$, we have $a(t) \leq t$
and hence $a_t(t) = a(t)$. Thus, it is enough to show that for every $i$ with $0 \leq i < t$, we have
\begin{equation}
  \label{toshow}
  \E \big[ Q_{t-T_{a_{i+1}(t)}}(S_{T_{a_{i+1}(t)}},y) \big] =
  \E \big[ Q_{t-T_{a_{i}(t)}}(S_{T_{a_{i}(t)}},y) \big] \,.
\end{equation}
Fix $i$ with $0 \leq i < t$.
We shall prove (\ref{toshow}) by conditioning on the values of
$T_i$ and $S_{T_i}$.
  That is, we will show that for every $j$ with $1 \leq j \leq t$
  and $S \subset V$, we have
\begin{equation}
  \label{toshowtwo}
  \E \big[ Q_{t-T_{a_{i+1}(t)}}(S_{T_{a_{i+1}(t)}},y) \given T_i = j, S_{T_i} = S \big] =
\E \big[ Q_{t-T_{a_{i}(t)}}(S_{T_{a_{i}(t)}},y)
\given T_i = j, S_{T_i} = S \big] 
  \,.
\end{equation}
So now fix $j$ and $S$, and suppose that $T_i = j$ and $S_{T_{i}} = S$.
Recall that $L_{i+1}$ is a function $S_{T_{i}}$.
If $T_{i+1} = j + L_{i+1} > t$ then 
$a_{i}(t) = a_{i+1}(t) = i$ and hence (\ref{toshowtwo}) is trivially true,
so assume $T_{i+1} \leq t$. Then $a_{i+1}(t) = i+1$
and $a_i(t) = i$.
So in this case 
we
just have to show that 
\begin{equation}
  \label{toshowthree}
  \E \big[ Q_{t-T_{i+1}}(S_{T_{i+1}},y) \given T_i = j, S_{T_i} = S \big] =
  Q_{t-j}(S ,y)  \,.
\end{equation}
But
\begin{eqnarray}
  \E \big[ Q_{t-T_{i+1}}(S_{T_{i+1}},y) \given T_i = j, S_{T_i} = S \big]   \nonumber
  &=&
      \sum_{x \in V}
      \prob \big[ x \in S_{T_{i+1}} \given T_i = j, S_{T_i} = S \big]
      Q_{t - T_{i+1}} (x, y)      \nonumber \\
  &=&  
      \sum_{x \in V}
      \prob \big[ x \in S_{T_{i+1}} \given T_i = j, S_{T_i} = S \big]
      \pi(x) p^{t - T_{i+1}} (x, y)     \nonumber  \\
  &=&
      \sum_{x \in V}
      Q_{L_{i+1}}(S, x) p^{t - T_{i+1}} (x, y),   \label{seven}
\end{eqnarray}
where the last line holds by the definition of the intermittant
evolving set process. Note that for each $s \in S$, we have
\begin{eqnarray}
  \sum_{x \in V}
  Q_{L_{i+1}}(s, x) p^{t - T_{i+1}} (x, y)
  &=&
                                              \sum_{x \in V}    \pi(s) p^{L_{i+1}}(s, x)
                                               p^{t - T_{i+1}}(x, y) \nonumber  \\
&=&
                                               \pi(s) 
                                               p^{t + L_{i+1}- T_{i+1}}(s, y) \nonumber   \\
&=&                                            
                                              Q_{t - T_{i}}(s, y),     \label{it}
\end{eqnarray}
where the last line holds because $T_{i+1} = T_i + L_{i+1}$. Finally, note that combining (\ref{it}) with (\ref{seven}) yields
(\ref{toshowthree}).
\end{proof}
\begin{corollary} \label{ProbPowers}
If the initial state of the intermittant evolving set process is $S=\{x_0\}$, then
	\begin{equation*}
	P^{t}(x_0,y) = \frac{1}{\pi(x_0)}\E \big[ Q_{t-T_{a(t)}}(S_{T_{a(t)}},y) \big].
	\end{equation*}
\end{corollary}

\section{Relating Transience and Evolving Sets}\label{RelatingTransienceandEvolvingSets}

Corollary \ref{ProbPowers} is useful since it will help us to prove an upper bound on  $\sum_{t\geq 0} P^t(x_0,y)$,
which implies transience. In Lemma \ref{transiencelemma} below we
relate this sum to $\sum_{i=0}^\infty \E \big[\sqrt{\pi(S_{T_i})}\big]$, for a specific choice of
$\pi$ and $(L_j)_{j\geq 1}$. Henceforth we will fix $\pi$ and $(L_j)_{j\geq 1}$ to be the following :
\begin{equation}
  \label{dagger}
	\pi(x) = \degree(x); \hspace{1 cm}
	L_m = 2 \bigg \lceil   {\log( 8 \cdot \pi (S_{T_{m-1}})) \over C}  \bigg \rceil \mbox{ for $m\geq 1$},
      \end{equation}
where $C$ is the constant appearing in the entopy growth condition.       
\begin{lemma}\label{transiencelemma}
  Let $\pi$ and $L_m$ be as in (\ref{dagger}). 
  Then for any $y\in V$, we have
	\begin{equation}
	\sum_{t=0}^\infty P^t(x_0,y)\leq 8d\bigg\lceil\frac{1}{C}\bigg\rceil \sum_{i=0}^\infty \E \Big[\sqrt{\pi(S_{T_i})}\Big].
	\end{equation}
\end{lemma}
\begin{proof}
	Note that for any $m\geq 0$, $S'\subset V$ and $y\in V$, we have $Q_m(S',y) \leq Q_m(V,y)=\pi(y)$. The equality in the preceding line is due to the fact that $\pi$ is a stationary measure. Moreover, $Q_m(\emptyset,y)=0$. Therefore, we have
	\begin{equation}\label{misc4}
		Q_{t-T_{a(t)}}(S_{T_{a(t)}},y) \leq Q_{t-T_{a(t)}}(V,y)\ind\big(S_{T_{a(t)}}\neq \emptyset\big) = \pi(y)\ind\big(S_{T_{a(t)}}\neq \emptyset\big).
	\end{equation}
	Using (\ref{misc4}) above and Corollary \ref{ProbPowers}, we get
	\begin{align}
		\sum_{t=0}^\infty P^t (x_0, y)& = \frac{1}{\pi(x_0)}\sum_{t=0}^\infty \E\big[Q_{t-T_{a(t)}}(S_{T_{a(t)}},y) \big]\nonumber\\
		&\leq \frac{\pi(y)}{\pi(x_0)}\sum_{t=0}^{\infty}\E\big[\ind\big(S_{T_{a(t)}}\neq \emptyset\big)\big]\nonumber\\
		&\leq d \sum_{t=0}^{\infty}\E\big[\ind\big(S_{T_{a(t)}}\neq \emptyset\big)\big],\label{misc5}
	\end{align}
	where the last inequality follows from the fact that $1\leq \pi(z)\leq d$ for any $z$. Next,
        note that when $T_{m}\leq t < T_{m+1}$, we have $a(t)=m$ and hence
	\begin{equation*}
		\sum_{t=0}^\infty \ind\big(S_{T_{a(t)}}\neq \emptyset\big)= \sum_{i=0}^\infty (T_{i+1}-T_i) \ind\big(S_{T_i}\neq \emptyset\big) = \sum_{i=0}^\infty L_{i+1} \ind\big(S_{T_i}\neq \emptyset\big).
	\end{equation*}
	Taking the value of $L_i$ from (\ref{dagger}) and combining with (\ref{misc5}) gives
	\begin{equation}\label{misc6}
          \sum_{t=0}^\infty P^t (x_0, y) \leq 
          d\cdot \E\Bigg[\sum_{i=0}^\infty 2\bigg\lceil\frac{\log(8\cdot\pi(S_{T_{i}}))}{C}\bigg\rceil \ind\big(S_{T_i}\neq \emptyset\big)\Bigg].
	\end{equation}
	Now observe the following fact about the natural logarithm which will be useful for bounding the above.
	\begin{equation}\label{misc3}
		4\sqrt{x}\geq \ceil{\log(8x)} \quad\text{for } x\geq1.
	\end{equation}
	To prove this, first note that $4\sqrt{x}\geq \log(8x)+1$ for $x\geq 1$. This is true since the
        inequality
        holds for $x=1$ and ${d \over dx}(4\sqrt{x}-\log(8x)-1) = {2 \over \sqrt{x}} - {1 \over x} \geq 0$ when $x\geq 1$.
        Since $\ceil{\log(8x)}\leq \log(8x)+1$, the previous
        inequality gives us (\ref{misc3}).\\
	Thus (\ref{misc6}) and (\ref{misc3}) imply that
	\begin{align*}
          \sum_{t=0}^\infty P^t (x_0, y)
          &\leq 8d\bigg\lceil\frac{1}{C} \bigg\rceil
          \E \bigg[\sum_{i=0}^\infty \sqrt{\pi(S_{T_i})}
          \ind\big(S_{T_i}\neq \emptyset\big)\bigg]
          \\
          &= 8d\bigg\lceil\frac{1}{C}\bigg\rceil \sum_{i=0}^\infty \E \Big[\sqrt{\pi(S_{T_i})} \Big].
	\end{align*}
\end{proof}

\section{Decay of $\E \sqrt{\pi(S_{T_i})}$}
We will show that $\E\sqrt{\pi(S_{T_i})}$ decays exponentially in $i$, by proving the following theorem.
\begin{theorem}\label{CondDecay}
For all $m \geq 1$, there exists a constant $0\leq \alpha <1$ depending only on $C$ and $d$ such that
\begin{equation}
\E\Big[\sqrt{\pi(S_{T_m})}\ \Big| \ S_{T_{m-1}} \Big] \leq  \alpha\cdot \sqrt{\pi (S_{T_{m-1}})}.
\end{equation}
\end{theorem}
We will come back to the proof of theorem \ref{CondDecay} after we prove a lemma.
\begin{lemma}
  \label{rootdecay}
  Let $R$ be a nonnegative random variable and suppose that $\E(R) = 1$. Then
  \[
    \E \sqrt{R}  \leq 1 - {1 \over 8} \E \Bigl(R \, \ind(R \geq 4)\Bigr) \,.
  \]
\end{lemma}  
\begin{proof}
Let $f(x) = \sqrt{x}$. 
Since $f$ is concave and $f'(1) = {1 \over 2}$, we have
  \[
    \sqrt{x} \leq 1 + {1 \over 2} (x-1) \;\; \mbox{for all $x \geq 0$.}
    \]
Since $f'(4) = {1 \over 4} < {1 \over 3}$, concavity of $f$ implies that
  \[
    \sqrt{x} \leq 1 + {1 \over 3} (x-1) \;\; \mbox{for all $x \geq 4$.}
    \]
    It follows that
    \[
      \sqrt{R} \leq 1 + {1 \over 2}(R-1) - {1 \over 6}(R-1) \ind(R \geq 4). 
      \]
      Taking expectations gives
      \begin{eqnarray*}
        \E \sqrt{R} &\leq& 1 - {1 \over 6} \E \Bigr( (R-1)\, \ind(R \geq 4) \Bigr) \\
                    &\leq& 1 - {1 \over 8} \E \Bigr( R \, \ind(R \geq 4) \Bigr),
      \end{eqnarray*}                    
      where the last line follows from the fact that
      \[
        (R-1) \, \ind(R \geq 4) \geq 
        {3 \over 4} R \, \ind(R \geq 4) \,.
      \]
      
\end{proof}
\begin{proofofthm}{\ref{CondDecay}}
  We consider one ``superstep'' of the evolving set process,
  that is, the update from time $T_m$ to time $T_{m+1}$. 
Suppose that 
$S_{T_m} = S$ and let $\St$ be the value of $S_{T_{m+1}}$, namely
\[
  \St =  \{y: Q_{L_m}(S,y)\geq U_m \pi(y)\},
\]
where $U_m$ is a uniform random variable. Define
$R := {\pi(\St) \over \pi(S)}$.
  Note that $\E R = 1$, since
  \begin{eqnarray*}
    \E \pi(\St) &=& \sum_{y \in V} \pi(y) \P(y \in \St) \\
                &=& \sum_{y \in V} \pi(y) {Q_{L_m}(S, y) \over \pi(y)} \\
                &=& Q_{L_m}(S, V) = \pi(S)
                    \;.     
  \end{eqnarray*}
  (This is analogous to the martingale property of the evolving sets of
  \cite{evolvingsets}.)
  Thus, we can apply lemma \ref{rootdecay}.
  With this in mind, we aim for a lower bound
  on $\E \Bigl( R \ind(R \geq 4) \Bigr)$.

  Let $S_u$ be the value of $\St$ when $U_m = u$. That is,
\[
  S_u =  \{y: Q_{L_m}(S,y)\geq u \pi(y)\}.
\]
Let $\sets$ be the collection of sets $\{ S_u : u \in [0,1]\}$.
Note that the collection $\sets$ 
is nested. Furthermore, note that
$\{S' \in \sets : \pi(S') < 4\pi(S) \}$
is a finite subset of $\sets$, so it has a largest element,
which we denote by $S_*$. This set $S_*$ has the
property that $\pi(S_*) < 4\pi(S)$, but any
set in $\sets$ that strictly contains $S_*$ has measure
at least $4 \pi(S)$. We have
\begin{eqnarray}
  \E \Bigl( R \ind(R \geq 4) \Bigr) &=&   \nonumber
     {1 \over \pi(S)} \E \Bigl( \pi(\St) \ind( \St \supsetneq S_*) \Bigr) \\
                        &\geq& {1 \over \pi(S)} \E( \pi( \St - S_*))   \label{abba}
\end{eqnarray}
But
\begin{eqnarray}
  \nonumber
  \E( \pi( \St - S_*))  &=& \int_0^1 \sum_{y \in S_*^c} \ind(y \in S_u) \pi(y) \,du \\
  &=& \int_0^1 \sum_{y \in S_*^c} \ind\Bigl( {Q_{L_m}(S, y) \over \pi(y)} \geq u \Bigr) \pi(y) \,du .   \label{f}
\end{eqnarray}
Since the indicator is nonnegative, we can apply Fubini and write the quantity (\ref{f})
as
\begin{eqnarray}
\sum_{y \in S_*^c}  \Bigl[ \int_0^1  \ind\Bigl( {Q_{L_m}(S, y) \over \pi(y)} \geq u \Bigr) \,du \Bigr] \pi(y)  &=&   \nonumber 
\sum_{y \in S_*^c}  \Bigl[ {Q_{L_m}(S, y) \over \pi(y)} \Bigr] \pi(y)  \\
                                              &=& Q_{L_m}(S, S_*^c)     \nonumber    \\
                      &\geq&
                  \pi(S) \frac{C L_m - \log(2|S_*|)}{L_m \log d},  \label{gap}
\end{eqnarray}
by Corollary \ref{maincor}. Now, using the fact that
$\pi(S_*) < 4 \pi(S)$ and the definition of $L_m$ gives
\begin{eqnarray*}
  \log(2|S_*|) &\leq&  \log(8 \pi(S)) \\
               &=& \half C \cdot {2 \log(8 \pi(S)) \over C} \\
               &\leq& \half C L_m  .
\end{eqnarray*}
It follows that the quantity (\ref{gap}) is at least
\[
  \pi(S) {C \over 2 \log d},
\]
Combining this with (\ref{abba}) gives
\[
  \E \Bigl( R \ind(R \geq 4) \Bigr) \geq {C \over 2 \log d} \,.
\]
Hence Lemma \ref{rootdecay} implies that the Theorem holds with the constant $\alpha = 1 - {C \over 16 \log d}$. 

\end{proofofthm}
By induction, we have the following corollary.
\begin{corollary}\label{maincor}
	Let $\pi,d,C, S_{T_j}$ be as in section \ref{EvolvingDef}. Then,
	\begin{equation}
	\E\Big[\sqrt{\pi(S_{T_m})} \Big] \leq \alpha^m\cdot \pi(x_0)
	\end{equation}
	where $\alpha < 1$ is a constant that depends only on $C$ and $d$. 
      \end{corollary}

The main theorem of this paper can now be easily proved using Lemma \ref{transiencelemma} and Corollary \ref{maincor} as follows.\\ \\
\begin{proofofthm}{\ref{mainthm}}
  To show that any vertex $x_0$ is transient, it is sufficient to show that $\sum_{t=0}^\infty P^t(x_0,x_0)<\infty$ . By Lemma \ref{transiencelemma} and
  Corollary \ref{maincor} we get
	\begin{align*}
	\sum_{t=0}^\infty P^t(x_0,x_0) &\leq 8d \bigg\lceil\frac{1}{C}\bigg\rceil \sum_{i=0}^\infty \E \Big[\sqrt{\pi(S_{T_i})}\Big]\\
	&\leq 8d \bigg\lceil\frac{1}{C}\bigg\rceil\sum_{i=0}^\infty \alpha^i \pi(x_0)<\infty
	\end{align*}
	since $0\leq\alpha<1$. 
\end{proofofthm}

\section{Necessity of the Uniformity of $C$ in the Linear Entropy Growth Condition} \label{Example}
The following example, which was conveyed to us by Gady Kozma,
shows that
the constant $C$ in  Theorem \ref{mainthm} 
cannot be allowed to depend on the starting position
of the walk. Start with $\Z^+$ and then add an edge from each vertex  
$n \geq 1$ to the root of $T_n$,
where $T_n$ is a full binary tree of height $2 \tower n$ (i.e., a tower of powers of height $n$). Call this graph $G$.
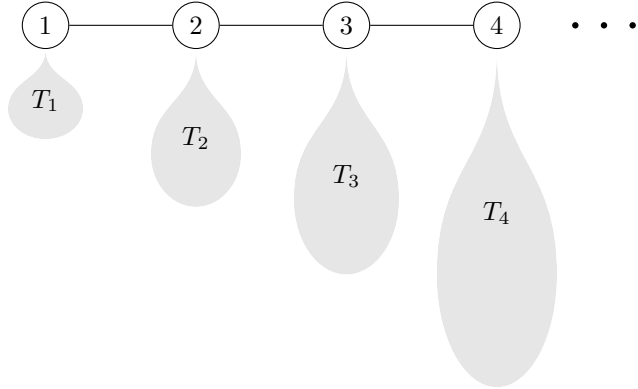
\begin{figure}[!h]
	\centering
	\begin{tikzpicture}
		\node[shape=circle,draw=black] (mid1) at (0,0) {$1$};
		\node[shape=circle,draw=black] (mid2) at (2,0) {$2$};
		\node[shape=circle,draw=black] (mid3) at (4,0) {$3$};
		\node[shape=circle,draw=black] (mid4) at (6,0) {$4$};
		\draw (mid1) -- (mid2)--(mid3)--(mid4); 
		\node at (7.5, 0) {\Huge $\dots$};
		\path (mid1.south) pic[xscale = 5, yscale = 4] {raindrop};
		\path (mid2.south) pic[xscale = 6, yscale = 7] {raindrop};
		\path (mid3.south) pic[xscale = 7, yscale = 10] {raindrop};
		\path (mid4.south) pic[xscale = 8, yscale = 15] {raindrop};
		\node at (0,-1) {$T_1$};
		\node at (2,-1.5) {$T_2$};
		\node at (4,-2) {$T_3$};
		\node at (6,-2.5) {$T_4$};
	\end{tikzpicture}  \\
	\caption{The Graph $G$}
\end{figure}
Then we will see that although the random walk on $G$ is recurrent, the entropy of the random walk grows at least linearly
for any starting point. 
First, note that simple random walk
on $G$,
which we will denote by $(X_t)_{t\geq 0}$,
is recurrent, since the effective resistance between any vertex and infinity is infinite. (See \cite{probontrees}
for background on random walks and electrical networks; or for an alternative argument, note that if 
$T_1 < T_2 < \cdots$ are the times
when the walk is at one of the vertices $n \in \Z^+$, then $X_{T_t}$ is a simple random walk on the positive integers, which is recurrent.) \\
\\
It remains to show is that for simple random walk on $G$, the entropy grows at least linearly.
For $n \geq 1$ let $V_n$ be the union of $\{n\}$ and vertex set of $T_n$, and define $S_n = V_1 \cup \cdots \cup
V_n$. Let $x_0$ be the starting state of the walk and define
\[
  l = \min\{n: x_0 \in S_n \} \,.
\]
We shall find a linear lower bound on the entropy of $X_t$ that holds whenever $t$ is
sufficiently large.
For $t \geq 8 |S_l|$, define
\[
  k_0 := \max \Big \{ k: 2|S_k|^2 \leq {t \over 4} \Bigl \}.
  \]
  A calculation shows that the heights of the trees $T_n$ grow sufficiently fast so that
  for sufficiently large $t$ we have $\diam(T_{k_0+2}) \geq t$.
Let $T_{k_0} = \min \{t: X_t = k_0 \}$ be  the hitting time of ${k_0}$. 
Since for any tree with $i$ vertices, the maximum expected  
hitting time of any vertex is at most $2i^2$ (see Chapter 10 of \cite{lpw}), 
we have
\[
\E T_{k_0} \leq 2 |S_{k_0}|^2 \leq {t \over 4} \,,
\]
and hence 
\[
\P\left( T_{k_0} \geq {t \over 2}\right) \leq {1 \over 2}
\]
by Markov's inequality. After the walk hits $k_0$ the probability that 
it goes to $k_0 + 2$ in the next two steps, and then crosses the edge into $T_{k_0+2}$
in the step after that,
is also bounded away from zero. Finally, the probability that
the remaining steps up to time $t$ are all spent in $T_{k_0 + 2}$
is also bounded away from zero, because when the walk is in  $T_{k_0 + 2}$, the 
distance from $k_0+ 2$ behaves like a ${2 \over 3} \uparrow, {1 \over 3} \downarrow$ 
random walk until the walk hits the $k_0+2$ again.
Putting this all together, we see that the probability that the walk
performs at least ${t \over 2} - 3$ steps in $T_{k_0 + 2}$ is bounded away from zero,
and hence the
probability that
$X_t$ is a vertex in $T_{k_0+2}$ at distance at least ${t \over 12}$ from $k_0+2$
is
bounded away from zero.

Since for any $d$ the walk is equally likely to be at any of the $2^d$ vertices at distance $d$ from $k_0+2$ in $T_{k_0+2}$,
this implies that if $t$ is sufficiently large then $E_t \geq Ct$
for a universal constant
$C$. Since $E_t > 0$ for all $t > 0$ there is a constant
$C_{x_0}$, depending on the starting state $x_0$, such that
\[
  E_t \geq C_{x_0} t
\]
for all $t \geq 0$. 

\section{Acknowledgments}
We are grateful to Itai Benjamini for sharing this problem with us and we thank Gady Kozma for sharing the counterexample
of Section \ref{Example}.

\bibliographystyle{unsrtnat}

\end{document}